\documentclass[11pt]{amsart}
\usepackage{latexsym}
\usepackage{amsfonts}
\usepackage{amsmath,amssymb}
\usepackage{color}

\newcommand{\field}[1]{\mathbb{#1}}
\newcommand{\C}{\field{C}}

\newcommand{\R}{\field{R}}

\newcommand{\id}[1]{\mathfrak{#1}}
\newcommand{\textswab}[1]{\id{#1}}

\newcommand{\ignore}[1]{}

\newtheoremstyle{s2}{9pt}{9pt}{\rm}{}{\bf}{.}{0.5em}{}
\theoremstyle{s2}
\newtheorem{defi}{Definition}[section]

\newtheorem{re}[defi]{Remark}

\newtheoremstyle{s1}{9pt}{9pt}{\it}{}{\bf}{.}{0.5em}{}
\theoremstyle{s1}
\newtheorem{lem}[defi]{Lemma}
\newtheorem{theo}[defi]{Theorem}
\newtheorem{co}[defi]{Corollary}
\newtheorem{pr}[defi]{Proposition}

\DeclareMathOperator{\rank}{rank}
\DeclareMathOperator{\Sing}{Sing}
\DeclareMathOperator{\mult}{mult}
\DeclareMathOperator{\codim}{codim}

\newcommand{\pa}[2]{\frac{\partial #1}{\partial #2}}

\title[Effective Whitney theorem for complex polynomial mappings of the plane]{Effective Whitney theorem for complex polynomial mappings of the plane} \makeatletter

\@addtoreset{equation}{section}
\makeatother
\author{M. \ Farnik \& Z. Jelonek \& M.A.S. Ruas}
\address[M. Farnik]{Instytut Matematyczny\\
Polska Akademia Nauk\\
\'Sniadeckich 8, 00-656 Warszawa, Poland}
\email{michal.farnik@gmail.com}

\address[Z. Jelonek]{Instytut Matematyczny\\
Polska Akademia Nauk\\
\'Sniadeckich 8, 00-656 Warszawa, Poland}
\email{najelone@cyf-kr.edu.pl}

\address[M.A.S. Ruas]{Departamento de Matem\'atica,
ICMC-USP, Caixa Postal 668, 13560-970 S\~ao Carlos, S.P., Brasil}
\email{maasruas@icmc.usp.br}


\subjclass{14 D 06, 14 Q 20.}
\thanks{The first two authors are partially supported by the grant of Narodowe Centrum Nauki, grant number 2015/17/B/ST1/02637, the third author is partially
supported by CNPq grant 305651/2011-0  and FAPESP grant
2014/00304-2}

\begin{document}

\begin{abstract}
 We describe the topology of a general polynomial mapping $F=(f, g):\C^2\to\C^2$, where $\deg f=d_1$, $\deg g=d_2$.
We show that the set $C(F)$ of critical points of $F$ is a smooth connected curve, which is topologically equivalent to a sphere with $\frac{(d_1+d_2-3)(d_1+d_2-4)}{2}$ handles and $d_1+d_2-2$ points removed. Moreover, the discriminant of $\Delta(F)=F(C(F))$ is a curve birationally equivalent to $C(F)$ which has only cusps and nodes as singularities and it has
$$c(F)=d_1^2+d_2^2+3d_1d_2-6d_1-6d_2+7$$
simple cusps and
$$d(F)=\frac{1}{2}\left[(d_1d_2-4)((d_1+d_2-2)^2-2)-(d-5)(d_1+d_2-2)-6\right]$$
nodes  (here $d = \gcd(d_1,d_2)$).
If $d_1=d_2$ then $\Delta(F)$ has $d_1+d_2-2$ smooth points at infinity, at which it is tangent to the line at infinity with multiplicity $d_1$. If $d_1>d_2$ then $\Delta(F)$ has exactly one singular point $P$ at infinity with $d_1+d_2-2$ branches and with delta invariant
$$\delta_P=\frac{1}{2}d_1(d_1-d_2)(d_1+d_2-2)^2+\frac{1}{2}(-2d_1+d_2+d)(d_1+d_2-2).$$

Finally let $F=(f, g):\C^2\to\C^2$ be an arbitrary polynomial mapping
with $\deg f\le d_1$ and $\deg g\le d_2.$ Assume that $F$  has
generalized cusps at points $a_1,\ldots, a_r$. Then $\sum^r_{i=1}
\mu_{a_i}\le d_1^2+d_2^2+3d_1d_2-6d_1-6d_2+7$, where $\mu_{a_i}$
denotes the index of a generalized cusp at the point $a_i$.
\end{abstract}

\maketitle

\bibliographystyle{alpha}

\section{Introduction}

\noindent In \cite{wh} Whitney showed that a general smooth
mapping $F: \R^2\to \R^2$ has only two-folds and simple cusps as
singularities. The problem of  counting the number of cusps of a
general perturbation of a plane-to-plane real singularity was
considered by Fukuda and Ishikawa in \cite{fi}. They proved that
the number modulo 2 of cusps of a general perturbation $F$ of a
finitely determined map-germ $F_0:(\mathbb R^{2},
0)\rightarrow(\mathbb R^{2},0)$ is a topological invariant of
$F_0$. More recently, in \cite{ks} Krzy\.zanowska and Szafraniec
gave an algorithm  to compute the number of cusps for sufficiently
general fixed real polynomial mapping of the real plane.

Algebraic formulas to count the number of cusps and nodes of a general
perturbation of a finitely determined holomorphic map-germ $F_0:
(\mathbb C^2,0) \to (\mathbb C^2,0)$,  were given by Gaffney and
Mond in \cite{gm1, gm2}. In this case any two general
perturbations $F$ of $F_0$ defined on a sufficiently small
neighborhood of $0$ are topologically equivalent, so the
number of cusps and nodes of $F$ is an invariant of the map-germ $F_0$.
 
Here we consider the global complex case. It is proved in
\cite{jel1} that two general polynomial mappings $F:\C^n\to \C^m$
(where $n\le m$) of fixed degree have the same topology. In
particular they have the same number of topological invariants. We
consider here the simplest case $F:\C^2\to\C^2$ and describe the topology of a general polynomial
mapping $F=(f, g):\C^2\to\C^2$, where $\deg f=d_1$, $\deg g=d_2$.
In Section \ref{secCF} we show that the set $C(F)$ of critical points of $F$
is a smooth connected curve which is topologically equivalent to
a sphere with $\frac{(d_1+d_2-3)(d_1+d_2-4)}{2}$ handles and
$d_1+d_2-2$ points removed. In Section \ref{secDF} we show that the discriminant of $\Delta(F)=F(C(F))$
is a curve birationally equivalent to $C(F)$ which has only cusps and nodes as singularities and it has
$$c(F)=d_1^2+d_2^2+3d_1d_2-6d_1-6d_2+7$$
simple cusps and
$$d(F)=\frac{1}{2}\left[(d_1d_2-4)((d_1+d_2-2)^2-2)-(d-5)(d_1+d_2-2)-6\right]$$
nodes  (here $d = \gcd(d_1,d_2)$).
If $d_1=d_2$ then $\Delta(F)$ has $d_1+d_2-2$ smooth points at infinity, at which it is tangent to the line at infinity with multiplicity $d_1$. If $d_1>d_2$ then $\Delta(F)$ has exactly one singular point $P$ at infinity with $d_1+d_2-2$ branches and with delta invariant $$\delta_P=
\frac{1}{2}d_1(d_1-d_2)(d_1+d_2-2)^2+\frac{1}{2}(-2d_1+d_2+d)(d_1+d_2-2).$$

We conclude the paper with Section \ref{secGC} where we introduce the notions of a generalized cusp and the index of a generalized cusp $\mu$ (see Definitions \ref{dfGenCus} and \ref{dfGenCusIn}). We show that if $F=(f, g):\C^2\to\C^2$ is an arbitrary polynomial mapping
with $\deg f\le d_1$, $\deg g\le d_2$ and generalized cusps at points $a_1,\ldots, a_r$ then $\sum^r_{i=1}
\mu_{a_i}\le d_1^2+d_2^2+3d_1d_2-6d_1-6d_2+7$.

Let us note that in the special case when $\gcd(d_1,d_2)=1$ the numbers $c(F)$ and $d(F)$ can be also computed by using local methods of Gaffney and Mond. Indeed, in that case we can choose a homogenous mapping $F\in \Omega(d_1,d_2)$, which is finitely determined and show that there is a deformation $F_t$ of a generic mapping to $F$ such that all cusps and nodes of $F_t$ tend to $0$ when $t \rightarrow 0$. In this case our formulas for $c(F)$ and $d(F)$ coincide with formulas of Gaffney-Mond. However in the general case this approach does not work since the appropriate homogenous mapping is not finitely determined.

\section{General polynomial mappings}
Let $\Omega_n(d_1,\ldots,d_n)$ denote the space   of polynomial
mappings $F:\C^n\to\C^n$ of multi-degree bounded by  $d_1,\ldots,
d_n$. Of course $\Omega_n(d_1,\ldots,d_n)$ has the structure of
the affine space. By $J^q(n)$ we denote the space of $q$-jets of
polynomial mappings $F=(f_1,\ldots, f_n): \C^n\to\C^n$. We define
it exactly as in \cite{gg}. However, if we fix coordinates in
the domain and the target then we can identify $J^q(n)$ with the
space $\C^n\times \C^n \times (\C^{N_q})^n$, where $\C^{N_q}$
parameterizes coefficients of polynomials of $n$-variables and of
degree bounded by $q$ with zero constant term (which correspond to
suitable Taylor polynomials). In further applications, in most
cases, we treat the space $J^q(n)$ in this simple way. In
particular for a given polynomial mapping $F: \C^n \to \C^n$ we
can define the mapping $j^q(F)$ as
$$j^q(F) : \C^n \ni x \mapsto \left(x, F(x),\left(\frac{\partial^{|\alpha|}{f_i}}{\partial x^\alpha}(x)\right)_{1\le
i\le n,1\le |\alpha|\le q}\right)\in J^q(n).$$

We start with the following observation:

\begin{pr}\label{generic}
Assume that $d_i\ge q$ for $i=1,\ldots,n$. Let $S_1,\ldots, S_k$ be
smooth algebraic  submanifolds of $J^q(n)$. Then there is a
Zariski open dense subset
$\Omega_n(d_1,\ldots,d_n)(S_1,\ldots,S_k)\subset \Omega_n(d_1,\ldots,d_n)$
such that for every $F\in \Omega_n(d_1,\ldots,d_n)(S_1,\ldots,S_k)$ we
have $$j^q(F)\pitchfork S_i, \ for \ i=1,\ldots,k.$$
\end{pr}

\begin{proof}
Consider the mapping $$\Psi: \Omega_n(d_1,\ldots,d_n)\times \C^n\ni
(F,x)\mapsto j^q(F)(x)\in J^q(n).$$ It is easy to see that $\Psi$
is a submersion. Fix $1\le i\le k$. By the transversality theorem
with a parameter the set of polynomials $F\in
\Omega_n(d_1,\ldots,d_n)$ such that $j^q(F)$ is transversal to $S_i$
is dense in $\Omega_n(d_1,\ldots,d_n)$. On the other side this set is
constructible in $\Omega_n(d_1,\ldots,d_n)$.
We conclude that there is a
Zariski open  dense subset $U_i\subset \Omega_n(d_1,\ldots,d_n)$ such
that for every $F\in U_i$ we have $j^q(F)\pitchfork S_i$. Now it
is enough to take
$\Omega_n(d_1,\ldots,d_n)(S_1,\ldots,S_k)=\bigcap^k_{i=1} U_i$.
\end{proof}

\begin{defi}
Let $S^1_k\subset J^1(n)$ denote the subvariety of $1$-jets of
corank $k$. Let $F\in \Omega_n(d_1,\ldots,d_n)$. We say that
$F$ is one-generic if $F$ is proper and $j^1(F)\pitchfork S^1_1$.
\end{defi}

By Proposition \ref{generic} the subset of one-generic mappings
contains a Zariski open dense subset of
$\Omega_n(d_1,\ldots,d_n)$. The following  is a well known result
(see for instance \cite{gg}). Our proof is based on Corollary 1.11
in \cite{jel}.

\begin{theo}\label{disc}
Let $F\in \Omega_n(d_1,\ldots,d_n)$ be  one-generic. Let $C(F)$
denote the set of critical points of $F$. Then there is an open
and dense subset $U\subset C(F)$ such that for every $a\in U$ the
germ $F_a: (\C^n,a)\to (\C^n, F(a))$ is holomorphically equivalent
to a two-fold.
\end{theo}

\begin{proof}
Let $\Delta=F(C(F))$ be the discriminant of $F$. Take $U=C(F)\setminus
F^{-1}(\Sing(\Delta))$. The set $U$ is a Zariski open dense subset of
$C(F)$. Take a point $a\in U$ and consider the germ $F_a:
(\C^n,a)\to (\C^n, F(a))$. By the choice of the point $a$ the germ
of the discriminant of $F_a$ is smooth. Hence by \cite{jel}, Corollary 1.11, the
germ $F_a$ is biholomorphically equivalent to a $k$-fold:
$(\C^n,0)\ni (x_1,\ldots, x_n)\mapsto (x_1^k,x_2,\ldots, x_n)\in
(\C^n,0)$. In particular ${\rm corank} [F_a]=1.$

Now note that
$J^1(n)\cong \C^n\times \C^n\times M(n,n)$, where $M(n,n)=\{
[a_{ij}], \ 1\le i,j\le n\}$ is the set of $n\times n$ matrices.
In these coordinates the set $S^1_1$ is given as $\{ (x,y,m) : \det
[m_{ij}]=\phi(x,y,m)=0\}$ on the open subset $\{ (x,y,m) :  {\rm corank} [m_{ij}]\le 1\}$. Since the
mapping $j^1(F)$ is transversal to $S^1_1$ the mapping $\phi\circ
j^1(F)=kx_1^{k-1}$ has to be a submersion at $0$. This is possible
only for $k=2$.
\end{proof}

For convenience of the reader we will give a short introduction to (the simplest case of) the Thom-Boardman singularities (see e.g. \cite{gg}):

\begin{defi}
We  define the set $S^k_1\subset J^k(n)$  as follows:

\begin{enumerate}
\item for $k=1$ it is the set of elements $\sigma=[F_a]$ such that ${\rm
corank}\ d_a F_a=1$,

\item $\sigma\in S^k_1$ if for a representant $[F_a]$ of
$\sigma$ we have $j^{k-1}(F_a)(a)\in S^{k-1}_1$, and $j^{k-1}(F_a)$ is transversal to $S^{k-1}_1$
and
$${\rm corank}\ d_a F_a|_{S^{k-1}_1(F_a)}=1.$$
\end{enumerate}
\end{defi}

We prove that all varieties $S^k_1$ are smooth (at least in all
cases which we need). Consequently for a mapping $F$ as above
the varieties $S^k_1(F)=j^k(F)^{-1}(S^k_1)$ are smooth, hence the
definition makes sense. In fact we show that on
a dense open subset of $J^q(n)$ the sets $S^k_1$ are described by algebraic
equations and they are locally closed smooth varieties.

\section{Plane mappings}\label{secCF}
Here we will study the set $\Omega_2(d_1,d_2)$.
Let us denote coordinates in $J^1(2)$ by
$$(x,y,f,g,f_x,f_y,g_x,g_y).$$ For a mapping
$F=(f,g)\in\Omega_2(d_1,d_2)$, we have $$j^1(F)=(x,y,f(x,y),g(x,y),
\frac{\partial{f}}{\partial{x}}(x,y),
\frac{\partial{f}}{\partial{y}}(x,y),
\frac{\partial{g}}{\partial{x}}(x,y),
\frac{\partial{g}}{\partial{y}}(x,y)),$$ which justifies our
notation. The set $S^1_1$ is given by the equation
$\phi(x,y,f,g,f_x,f_y,g_x,g_y)=f_xg_y-f_yg_x=0$. Since $S^1_1$
describes elements of rank one it is easy to see that it is a
smooth (non-closed) subvariety of $J^1(2)$.

Now we would like to describe the set $S^2_1$. We restrict our
attention only to  sufficiently general mappings. In the space
$J^2(2)$ we introduce coordinates $$(x,y,f,g,f_x,f_y,g_x,g_y,
f_{xx}, f_{yy}, f_{xy}, g_{xx}, g_{yy}, g_{xy}).$$ A generic
mapping $F$ satisfies $\rank d_a F\ge 1$ for every $a$ (because
$\codim S^1_2=4$). We can assume that $F=(f,g)$ and $\nabla_a
f\not=0$. The critical set of $F$ is exactly the set $S^1_1(F)$
and it has a reduced equation
$\frac{\partial{f}}{\partial{x}}(x,y)\frac{\partial{g}}{\partial{y}}(x,y)-
\frac{\partial{f}}{\partial{y}}(x,y)\frac{\partial{g}}{\partial{x}}(x,y)=0$,
which by simplicity we write as $f_x g_y-f_y g_x=0$. In particular
the tangent line to $S^1_1(f)$ is given as
$$(f_{xx}g_y+f_xg_{xy}-f_{xy}g_x-f_yg_{xx})v+(f_{xy}g_y+f_xg_{yy}-f_{yy}g_x-f_yg_{xy})w=0.$$
Consequently the condition for $[F_a]\in S^2_1$ is:
$$f_xg_y-f_yg_x=0$$ and
$$(f_{xx}g_y+f_xg_{xy}-f_{xy}g_x-f_yg_{xx})f_y-(f_{xy}g_y+f_xg_{yy}-f_{yy}g_x-f_yg_{xy})f_x=0.$$
Let us note that the last equation contains terms $g_{xx}f_y^2$
and $g_{yy}f_x^2$ hence for $\nabla f\not=0$ these two equations
form a complete intersection. In general, if we omit the
assumption $\nabla f\not=0$ the set $S^2_1$ is given in $J^2(2)$
by three equations:
$$L_1:=f_xg_y-f_yg_x=0,$$
$$L_2:=(f_{xx}g_y+f_xg_{xy}-f_{xy}g_x-f_yg_{xx})f_y-(f_{xy}g_y+f_xg_{yy}-f_{yy}g_x-f_yg_{xy})f_x=0,$$
and
$$L_3:=(f_{xx}g_y+f_xg_{xy}-f_{xy}g_x-f_yg_{xx})g_y-(f_{xy}g_y+f_xg_{yy}-f_{yy}g_x-f_yg_{xy})g_x=0.$$
As above by symmetry the set $S^2_1$ is smooth and locally is
given as a complete intersection of either $L_1, L_2$ or $L_1,
L_3$.

We will denote by $J, J_{1,1}, J_{1,2}$ curves given by $L_1\circ
j^1(F)=0$, $L_2\circ j^2(F)=0$ and $L_3\circ j^2(F)=0$. We will also
identify these curves with their equations.

\begin{defi}\label{gen}
Let $F\in \Omega_2(d_1, d_2)$. We say that $F$ is generic if $F$
is proper, $j^1(F)\pitchfork S^1_1, j^2(F)\pitchfork S^2_1,$  and
additionally $j^1(F)\pitchfork S^1_2$.
\end{defi}

Again by Proposition \ref{generic} the subset of generic mappings
contains a Zariski open dense subset of $\Omega_2(d_1, d_2)$. Thus
a general mapping is generic.

\begin{defi}
Let $F : (\C^2,a)\to (\C^2,F(a))$ be a holomorphic mapping. We say
that $F$ has a simple cusp at $a$ if $F$ is biholomorphically
equivalent to the mapping $(\C^2,0)\ni (x,y)\mapsto (x, y^3+xy)\in
(\C^2,0)$.
\end{defi}

For the convenience of the reader we give a precise description of
singularities of a generic (in the sense of Definition \ref{gen})
plane mapping:

\begin{theo}
Let $F:\C^2\to\C^2$ be a generic polynomial mapping. Then $F$ has
only two-folds and simple cusps as singularities.
\end{theo}

\begin{proof}
Since $F$ is generic, $\rank F\ge 1$. Since $j^1(F)\pitchfork
S^1_1$, the set $S^1_1(F)$ is smooth. Because $\rank F\ge 1$ we
have $S^1_1(F)=C(F)$ i.e. the set of critical values of $F$. For a
point $a\in C(F)$ we have two possibilities:

\begin{enumerate}
\item $\rank  d_a F|_{C(F)}=1$,
\item $\rank  d_a F|_{C(F)}=0$.
\end{enumerate}

\noindent In the case (1) we see that the discriminant
$F_a(C(F_a))$ of a germ $F_a$ is smooth and we can proceed as in
the proof of Theorem \ref{disc}. Hence $F_a$ is two-fold. Now
assume (2). We can assume that $F_a=(x, g(x,y))$. We have $L_1\circ
j^1(F_a)={g_y}(a)=0$ and $L_2\circ j^2(F)=g_{yy}(a)=0$. Since $F$
is a generic mapping, the curves $g_y=0$ and $g_{yy}=0$ intersect
transversally at $a$. In particular we have
$$\det \left[ \begin{array}{ccccccccc}
            g_{yx}(a) & g_{yy}(a) \\
            g_{yyx}(a) & g_{yyy}(a) \\
            \end{array}\right] \not=0.$$
Since $g_{yy}(a)=0$ it is easy to see that
$g_{yyy}(a)\not=0$ and $g_{xy}(a)\not=0$. We can assume that
$a=(0,0)$ and $g(0,0)=0$. By the Weierstrass Preparation Theorem
we have $$g(x,y)=v(x,y)(y^3+a_1(x)y^2+a_2(x)y+a_3(x)),$$ where $v,
a_i$ are holomorphic and $v(0,0)\not=0$, $a_i(0)=0$ for $i=1,2,3$.
By the Rouche Theorem we see that $\mu(F_a)=\mu((x,
y^3+a_1(x)y^2+a_2(x)y+a_3(x)))=3$ (here $\mu(F)$ denotes the
topological degree of the mapping $F$). Hence $F_a$ is a local
analytic covering of degree~$3$. By \cite{gun}, Theorem 12, p.
104, we have
$$(*)\quad y^3+b_1(x,g)y^2+b_2(x,g)y+b_3(x,g)=0.$$
Here $b_i$ are holomorphic
and $b_i(0,0)=0$. Now we  follow \cite{gg}, p. 148. The equality
$(*)$ can be rewritten as
$$(**)\quad (y+b_1(x,g)/3)^3+c(x,g)(y+b_1(x,g)/3)+d(x,g)=0,$$
where $c,d$ are
holomorphic and $c(0,0)=d(0,0)=0.$ Since $g(0,y)=ay^3+\ldots$, we
have from $(**)$ that $\frac{\partial{d}}{\partial y}(0,0)\not=0$.
The inequality $\frac{\partial{g}^2}{\partial x\partial
y}(0,0)\not=0$ implies $\frac{\partial c}{\partial x}(0,0)\not=0$.
Moreover $(**)$ implies $\frac{\partial{d}}{\partial x}(0,0)=0$.
Take the coordinates
$$x_1=c(x,g), \ y_1=y+b_1(x,g)/3~~~~~~{\rm and} ~~~~x_2=c(x,y), \ y_2= -d(x,y).$$
We have $F_a^*(x_2,y_2)=(c(x,g), -d(x,g))=(c(x,g),
(y+b_1(x,g)/3)^3+c(x,g)(y+b_1(x,g)/3))=(x_1, y_1^3+x_1y_1)$.
\end{proof}

Now we compute the number of cusps of a general polynomial mapping
$F\in \Omega_2(d_1,d_2)$. To do this we need the series of lemmas:

\begin{lem}\label{infty0}
Let $L_\infty$ denote the line at infinity of
$\C^2$. There is a non-empty open subset $V\subset
\Omega_2(d_1,d_2)$ such that for all $(f,g)\in V:$

\begin{enumerate}
\item $\left\{\frac{\partial f}{\partial x}=0\right\}\pitchfork \left\{\frac{\partial
f}{\partial y}=0\right\}$,
\item $\overline{\left\{\frac{\partial f}{\partial x}=0\right\}}\cap \overline{\left\{\frac{\partial
f}{\partial y}=0\right\}}\cap L_\infty=\emptyset$.
\end{enumerate}

\end{lem}

\begin{proof}
The case $d_1=1$ is trivial so assume $d_1>1$. Let us note that
the set $S\subset J^1(2)$ given by $\{ f_x=f_y=0\}$ is smooth.
Hence (1) follows from Proposition \ref{generic}. To prove (2) it is
enough to assume that $f\in H_d$, where $H_d$ denotes the set of
homogenous polynomials of two variables of degree $d$.  Let $\Psi
: H_d\times (\C\times \C)\setminus \{0,0\}\ni (f, x,y)\mapsto
(\frac{\partial f}{\partial x}(x,y), \frac{\partial f}{\partial
y}(x,y))\in \C^2$. It is easy to see that $\Psi$ is a submersion.
Indeed, if $f=\sum a_i x^{d-i}y^i$ then $f_x:=\frac{\partial
f}{\partial x}(x,y)=da_0x^{d-1}+\ldots+a_{d-1}y^{d-1},
f_y:=\frac{\partial f}{\partial
y}(x,y)=a_1x^{d-1}+\ldots+da_dy^{d-1}$. Since $(x,y)\not=(0,0)$ we
can assume by symmetry that $y\not=0$. Now
$\pa{f_x}{a_{d-1}}=y^{d-1}, \pa{f_x}{a_d}=0,
\pa{f_y}{a_d}=dy^{d-1}$. Thus
$\pa{(f_x,f_y)}{(a_{d-1},a_d)}=dy^{2(d-1)}\not=0$.

Hence  for a general polynomial $f\in H_d$ the mapping $\Psi_f:
 (\C\times \C)\setminus \{0,0\}\ni (x,y)\mapsto (\frac{\partial f}{\partial
x}(x,y), \frac{\partial f}{\partial y}(x,y))\in \C^2$ is
transversal to the point $(0,0)$. In particular $\Psi_f^{-1}(0,0)$
is either zero-dimensional or the empty set. Since $f$ is a
homogenous polynomial the first possibility is excluded. This
means that  $\overline{\left\{\frac{\partial f}{\partial
x}=0\right\}}\cap \overline{\left\{\frac{\partial f}{\partial
y}=0\right\}}\cap L_\infty=\emptyset.$
\end{proof}

\begin{lem}\label{infty}
Let $L_\infty$ denote the line at infinity of
$\C^2$. There is a non-empty open subset $V\subset \Omega_2(d_1,d_2)$ such
that for all $F=(f,g)\in V$:

\begin{enumerate}
\item $\overline{J(F)}\cap \overline{J_{1,1}(F)}\cap L_\infty=\emptyset$,
\item $\overline{J(F)}\pitchfork L_\infty$.
\end{enumerate}
\end{lem}

\begin{proof}
Since the case $d_1=d_2=1$ is trivial and the assertion does not depend on replacing $(f,g)$
with $(g,f)$ we may assume that $d_2>1$.
We consider the (general) case when $\deg f=d_1$ and $\deg g=d_2$.
Hence $\overline{J(F)}\cap L_\infty$ and $\overline{J_{1,1}(F)}\cap L_\infty$ depend only on the homogeneous parts of $f$ and $g$ of degree
$d_1$ and $d_2$ respectively. Let $H_d$ denote the set of
homogeneous polynomials of degree $d$ in two variables. It is
sufficient to show that there is an open subset $V\subset
H_{d_1,d_2}:=H_{d_1}\times H_{d_2}$ such that $\overline{J(F)}\cap \overline{J_{1,1}(F)}\cap L_\infty=\emptyset$ for all $F=(f,g)\in V$.

Consider the set $X=\left\{(p,F)\in\mathbb{P}^1\times H_{d_1,d_2}\
:\ J(F)(p)=J_{1,1}(F)(p)=0\right\}$. Note that $X$ is a
closed subset of $\mathbb{P}^1\times H_{d_1,d_2}$, and if
$\overline{J(F)}\cap \overline{J_{1,1}(F)}\cap L_\infty\neq\emptyset$ then
$F$ belongs to the image of the projection of $X$ on
$H_{d_1,d_2}$. So to prove (1) it is sufficient to show that $X$
has dimension strictly smaller than the dimension of
$H_{d_1,d_2}$.

Let $q=(1:0)\in \mathbb{P}^1$, $Y:=\{q\}\times
H_{d_1,d_2}$ and $X_0=X\cap Y$. Note that all fibers of the projection
$X\rightarrow\mathbb{P}^1$ are isomorphic to $X_0$. Thus
$\dim(X)=\dim(X_0)+\dim(\mathbb{P}^1)$ and to prove (1) it is sufficient to
show that $X_0$ has codimension at least $2$ in $Y$.

Let $p=(q,F)\in Y$ and let $a_i$ and $b_i$ be the parameters in $H_{d_1,d_2}$ giving
respectively the coefficients of $f$ at $x_1^{d_1-i}x_2^i$ and
of $g$ at $x_1^{d_2-i}x_2^i$. For $0\leq i+j\leq d_1$, we have
$\frac{\partial^{i+j}f}{\partial
x_1^ix_2^j}(q)=\frac{(d_1-j)!j!}{(d_1-i-j)!}a_j(F)$ and similarly
for $g$ and $b_j$.

To conclude the proof of (1) we will show that the codimension of $\{a_0=0\}\cap X_0$ in $Y$ is at least $2$ and $\nabla J$ and $\nabla J_{1,1}$ are
linearly independent outside $\{a_0=0\}\cap X_0$ and thus the variety $X_0$ has
codimension $2$ in $Y$.

Let us calculate $J(p)$. We have $J(p)=(f_xg_y-f_yg_x)(q,F)=(d_1a_0b_1-d_2a_1b_0)(F)$.
Thus $\{a_0=0\}\cap X_0\subset \{a_0=a_1b_0=0\}\cap Y$ has codimension at least $2$ and we may assume $a_0(F)\neq 0$ in further calculations.

We have $\pa{J}{b_1}(p)=\pa{d_1a_0b_1-d_2a_1b_0}{b_1}(F)=d_1a_0(F)$ and $\frac{\partial J(p)}{\partial b_2}=0$.
Now let us calculate $\pa{J_{1,1}}{b_2}(p)$. The
coefficient $b_2$ can only be obtained from $\frac{\partial^2
g}{\partial x_2^2}$, which is present in $J_{1,1}$ in the
summand $-2\frac{\partial^2 g}{\partial
x_2^2}(d_1\pa{f}{x_1})^2$. Thus
$\pa{J_{1,1}}{b_2}(p)=\pa{(-2d_1^2b_2a_0^2)}{b_2}(F)=-2(d_1a_0(F))^2$.
So $\det\pa{(J,J_{1,1})}{(b_1,b_2)}(p)=-2(d_1a_0(F))^3\neq 0$.

To prove (2) note that $\overline{\big\{\pa{J}{x}(F)=0\big\}}\cap\overline{\big\{\pa{J}{y}(F)=0\big\}}\subset \overline{J_{1,1}(F)}$, hence (1) implies (2).
\end{proof}

\begin{lem}\label{infty1}
There is a non-empty open subset $V_1\subset \Omega_2(d_1,d_2)$
such that for all $(f,g)\in V_1$ and every $a\in \C^2$: if
$\frac{\partial{f}}{\partial{x}}(a)=0$ and
$\frac{\partial{f}}{\partial{y}}(a)=0$, then
$\frac{\partial{g}}{\partial{x}}(a)\not=0$ and
$\frac{\partial{g}}{\partial{y}}(a)\not=0$.
\end{lem}

\begin{proof}
Let us consider two subsets in $J^1(2)$: $S_1:=\{
(x,y,f,g,f_x,f_y,g_x,g_y): f_x=0,f_y=0, g_x=0\}$ and $S_2:=\{
(x,y,f,g,f_x,f_y,g_x,g_y): f_x=0,f_y=0, g_y=0\}$. By Proposition
\ref{generic} there is a non-empty open subset $V_1\subset
\Omega_2(d_1,d_2)$ such that for every $F\in V_1$ the mapping
$j^1(F)$ is transversal to $S_1$ and $S_2$. Since these subsets
have codimension three, we see that the image of $j^1(F)$ is
disjoint with $S_1$ and $S_2$.
\end{proof}

\begin{lem}\label{infty3}
There is a non-empty open subset $V_2\subset \Omega_2(d_1,d_2)$
such that for all $(f,g)\in V_2$ we have
$\big\{\frac{\partial{f}}{\partial{x}}=0\big\} \cap
\big\{\frac{\partial{f}}{\partial{y}}=0\big\}\cap J_{1,2}(f,g)=\emptyset$.
\end{lem}

\begin{proof}
Let us consider the (non-closed) subvariety $S\subset J^2(2)$
given by equations: $f_x=0$, $f_y=0$,
$(f_{xx}g_y+f_xg_{xy}-f_{xy}g_x-f_yg_{xx})g_y-(f_{xy}g_y+f_xg_{yy}-f_{yy}g_x-f_yg_{xy})g_x=0$,
$g_x\not=0$, $g_y\not=0$. It is easy to check that $S$ is a smooth
complete intersection and it has codimension three. The set of
generic mappings $F$ which are transversal to $S$ contains a
Zariski open dense subset $V_2\subset \Omega_2(d_1,d_2)$. By
construction for all $(f,g)\in V_2$ we have
$\big\{\frac{\partial{f}}{\partial{x}}=0\big\} \cap
\big\{\frac{\partial{f}}{\partial{y}}=0\big\}\cap J_{1,2}(f,g)=\emptyset$.
\end{proof}

\begin{lem}\label{infty4}
There is a non-empty open subset $V_3\subset \Omega_2(d_1,d_2)$
such that for all $(f,g)\in V_3$ the curve $J(f,g)$ is transversal
to the curve $J_{1,1}(f,g)$.
\end{lem}

\begin{proof}
There is a Zariski open subset $V_3$ which contains only generic
mappings which satisfy hypotheses of all lemmas above. We can also
assume that the curves $\big\{\frac{\partial{f}}{\partial{x}}=0\big\}$ and
$\big\{\frac{\partial{f}}{\partial{y}}=0\big\}$ intersect transversally.
We have to show that the curves $J(f,g)$ and
 $J_{1,1}(f,g)$
intersect transversally at every point $a\in J(f,g)\cap
 J_{1,1}(f,g)$. If $\nabla
f\not=0$ then it follows from transversality of the mapping $F$
to the set $S^2_1$. Hence we can assume
$\big\{\frac{\partial{f}}{\partial{x}}(a)=0\big\}$ and
$\big\{\frac{\partial{f}}{\partial{y}}(a)=0\big\}$. By Lemma \ref{infty1}
we have  $\frac{\partial{g}}{\partial{x}}(a)\not=0$ and
$\frac{\partial{g}}{\partial{y}}(a)\not=0$. Let us denote:
$\frac{\partial f}{\partial x}(x,y)=f_x, \frac{\partial
f}{\partial y}(x,y)=f_y$, etc.  It is enough to prove that in
the ring ${\mathcal O}_a^2$ we have the
 equality $I=(f_xg_y-f_yg_x, (f_{xx}g_y+f_xg_{xy}-f_{xy}g_x-f_yg_{xx})f_y-(f_{xy}g_y+f_xg_{yy}-f_{yy}g_x-f_yg_{xy})f_x)={\textswab m}_a$,
 where ${\textswab m}_a$ denotes the maximal ideal of  ${\mathcal O}_a^2$.
Put $L=f_xg_y-f_yg_x$. Hence $I=(L,L_xf_y-L_yf_x)$. Since
$g_x(a)\not=0,g_y(a)\not=0$, we have
$$I=(L,g_x[L_xf_y-L_yf_x],g_y[L_xf_y-L_yf_x])=(L,L_xg_xf_y-L_yg_xf_x,L_xg_yf_y-L_yg_yf_x)=$$$$=(L,L_xg_yf_x-L_yg_xf_x,L_xg_yf_y-L_yg_xf_y)=
(L,f_x[L_xg_y-L_yg_x],f_y[L_xg_y-L_yg_x]).$$ By Lemma \ref{infty3}
we have $[L_xg_y-L_yg_x](a)\not=0$, hence $I=(f_x,f_y)={\textswab
m}_a$.
\end{proof}

Now we are in a position to prove:

\begin{theo}\label{thmcusps}
There is a Zariski open, dense subset $U\subset \Omega_2(d_1,d_2)$
such that for every mapping $F\in U$ the mapping $F$ has only
two-folds and cusps as singularities and the number of cusps is
equal to
$$d_1^2+d_2^2+3d_1d_2-6d_1-6d_2+7.$$
Moreover, if $d_1>1$ or $d_2>1$ then the set $C(F)$ of critical
points of $F$ is a smooth connected curve, which is topologically
equivalent to a sphere with $g=\frac{(d_1+d_2-3)(d_1+d_2-4)}{2}$
handles and $d_1+d_2-2$ points removed.
\end{theo}

\begin{proof}
If $d_1=d_2=1$ then the theorem is obvious. Hence we can assume that
$d_1>1$. Assume first that also $d_2>1$. Note that every point $a$
of the intersection of curves $J(f,g)$ and $J_{1,1}(f,g)$ with
$\nabla_a f\not=0$ is a cusp. Moreover for a general
mapping $F$ points with $\nabla_a f=0$ are not cusps (Lemma
\ref{infty3}). By Bezout Theorem we have that in $J(f,g)\cap J_{1,1}(f,g)$ there are exactly
$(d_1-1)^2$ points with $\nabla f=0$ and
that the number of cusps of a general mapping is equal to
$$(d_1+d_2-2)(2d_1+d_2-4)-(d_1-1)^2=d_1^2+d_2^2+3d_1d_2-6d_1-6d_2+7.$$
If $d_2=1$ then we can modify the definition of $J^2(2)$ (and
other spaces) and replace it by  its subspace given by equations
$g_{xx}=g_{xy}=g_{yy}=0$. Now again all  submersions considered by us
 remain to be submersions and we can proceed as above.
We leave the details to the reader.

Finally by Lemma \ref{infty} we have that $C(F)=S^1_1(F)$ is a
smooth affine curve which is transversal to the line at infinity.
This means that $\overline{C(F)}$ is also smooth at infinity,
hence it is a smooth projective curve of degree $d=d_1+d_2-2$.
Hence  by the Riemmann-Roch Theorem the curve $\overline{C(F)}$
has genus $g=\frac{(d-1)(d-2)}{2}$. This means in particular that
$\overline{C(F)}$ is homeomorphic to a sphere with
$g=\frac{(d-1)(d-2)}{2}$ handles. Moreover, by the Bezout Theorem
it has precisely $d$ points at infinity.
\end{proof}

\section{The discriminant}\label{secDF}
Here we analyze the discriminant of a general mapping from $\Omega(d_1,d_2)$. Let us recall that the discriminant of the mapping
$F:\C^2\to\C^2$ is the curve $\Delta(F):=F(C(F))$, where $C(F)$ is the critical curve of $F.$ We have:

\begin{lem}\label{bir}
There is a non-empty open subset $U\subset \Omega_2(d_1,d_2)$
such that for every mapping $F\in U$:
\begin{enumerate}
\item $F_{|C(F)}$ is injective outside a finite set,
\item if $p\in\Delta(F)$ then $|F^{-1}(p)\cap C(F)|\leq 2$,
\item if $|F^{-1}(p)\cap C(F)|= 2$ then the curve $\Delta(F)$ has a normal
crossing at $p$.
\end{enumerate}
\end{lem}

\begin{proof}
We may assume that $d_1\geq d_2$. Let $\Omega_2^*(d_1,d_2)$ be the set of $(f,g)\in\Omega_2(d_1,d_2)$ such that $g-g(0,0)$ is not $0$ and does not divide $f-f(0,0)$. Note that $\Omega_2^*(d_1,d_2)$ is a non-empty open subset of $\Omega_2(d_1,d_2)$ and if $F\in\Omega_2^*(d_1,d_2)$ and $\alpha\in\Omega_2(1,1)$ is an affine automorphism of $\C^2$ then $F\circ\alpha\in\Omega_2^*(d_1,d_2)$.

To prove (1) consider the set $X=\{(p,q,F)\in\C^2\times\C^2\times\Omega_2^*(d_1,d_2):\ p\neq q,\ F(p)=F(q),\ J(F)(p)=J(F)(q)=0\}$. We will show that $X$ has dimension not greater than $\dim \Omega_2^*(d_1,d_2)$. So the projection of $X$ on $\Omega_2^*(d_1,d_2)$ has finite fibers on some open subset $U\subset \Omega_2^*(d_1,d_2)$. Moreover if the fiber over $F$ is finite then $F_{|C(F)}$ is injective outside a finite set given by the fiber.

Let $p=(0,0)$, $q=(0,1)$, $Y:=\{p\}\times\{q\}\times\Omega_2^*(d_1,d_2)$
and $X_0=X\cap Y$. Note that all fibers of the projection
$X\rightarrow\C^2\times\C^2$ are isomorphic to $X_0$. Thus
$\dim(X)=\dim(X_0)+4$ and to prove (1) it is sufficient to
show that $X_0$ has codimension at least $4$ in $Y$.

Let $(p,q,F)\in Y$ and let $a_{ij}$ and $b_{ij}$ be the parameters in $\Omega_2(d_1,d_2)$ giving
respectively the coefficients of $f$ and $g$ at $x^iy^j$ . For $0\leq i+j\leq d_1$, we have
$\frac{\partial^{i+j}f}{\partial x^iy^j}(p)=i!j!a_{ij}(F)$
and $\frac{\partial^{i+j}f}{\partial x^iy^j}(q)=i!\sum_{k=j}^{d_1-i} \frac{k!}{(k-j)!}a_{ij}(F)$
and similarly for $g$ and $b_{ij}$.

The condition $F(p)=F(q)$ yields the equations $w_1=\sum_{j=1}^{d_1}a_{0j}(F)=0$ and $w_2=\sum_{j=1}^{d_2}b_{0j}(F)=0$, the conditions $J(F)(p)=J(F)(q)=0$ give $w_3=(a_{10}b_{01}-a_{01}b_{10})(F)=0$ and $w_4=(\sum_{j=0}^{d_1-1}a_{1j}\sum_{j=1}^{d_2}jb_{0j}-\sum_{j=1}^{d_1}ja_{0j}\sum_{j=0}^{d_2-1}b_{1j})(F)=0$. If $d_2\geq 2$ then note that the matrix $\pa{(w_1,w_2,w_3,w_4)}{(a_{01},b_{01},a_{10},a_{11})}$ is triangular and its determinant is equal to $b_{01}(F)\sum_{j=1}^{d_2}b_{0j}(F)$. Calculating similar derivations with $a_{10}$ replaced by $b_{10}$ or $a_{11}$ replaced by $b_{11}$ we obtain that $\nabla w_1,\ldots,\nabla w_4$ are independent outside $S=\{a_{01}(F)=b_{01}(F)=0\}\cup\{\sum_{j=1}^{d_2}a_{0j}(F)=\sum_{j=1}^{d_2}b_{0j}(F)=0\}$. Thus $X_0\setminus S$ has codimension $4$ in $Y$ and it is easy to see that $X_0\cap S$ has also codimension at least $4$.

If $d_2=1$ then we have $w_2=b_{01}(F)=0$. Since $F=(f,g)\in\Omega_2^*(d_1,1)$ we have $b_{10}(F)\neq 0$ and $x$ does not divide $f-f(0,0)$, i.e. $a_{0j}(F)\neq 0$ for some $j\geq 1$. Moreover we may take $w_3=a_{01}(F)=0$ and $w_4=\sum_{j=1}^{d_1}ja_{0j}(F)=0$. If $d_1=2$ then we obtain a contradiction thus showing that $X$ is in fact empty. If $d_1\geq 3$ then calculating $\det\pa{(w_1,w_2,w_3,w_4)}{(b_{01},a_{01},a_{02},a_{03})}=1$ we obtain that $X_0$ has codimension $4$.

To prove (2) consider the set $X=\{(p,q,r,F)\in\C^2\times\C^2\times\C^2\times\Omega_2^*(d_1,d_2):\ p\neq q\neq r\neq p,\ F(p)=F(q)=F(r),\
J(F)(p)=J(F)(q)=J(F)(r)=0\}$. Similarly as in (1) we compute that $X$ has codimension at least $7$. It follows that the projection of $X$ on $\Omega_2^*(d_1,d_2)$ has empty fibers on some open subset $U\subset \Omega_2^*(d_1,d_2)$. Note that unlike in (1) there are two types of fibers of the projection onto $\C^6$: $X_0:=\{((0,0),(1,0),(0,1))\}\times \Omega_2^*(d_1,d_2)\cap X$ and $X_t:=\{((0,0),(0,1),(0,t))\}\times \Omega_2^*(d_1,d_2)\cap X$. In both cases the computation is purely technical and similar to the computation in (1), so we leave the details to the reader.

To prove (3) note that if $q\in C(F)$ then $d_qF(T_qC(F))$ is spanned by the vector $(J_{1,1}(F)(q),J_{1,2}(F)(q))$. Thus if $F^{-1}(p)\cap C(F)=\{q_1,q_2\}$ then $\Delta(F)$ has a normal crossing at $p$ if and only if $(J_{1,1}(F)(q_1),J_{1,2}(F)(q_1))$ and $(J_{1,1}(F)(q_2),J_{1,2}(F)(q_2))$ are independent, i.e. $J_{1,1}(F)(q_1)J_{1,2}(F)(q_2)-J_{1,2}(F)(q_1)J_{1,1}(F)(q_2)\neq 0$. Similarly as in (1) let us consider the set $X=\{(p,q,F)\in\C^2\times\C^2\times\Omega_2^*(d_1,d_2):\ p\neq q,\ F(p)=F(q),\ J(F)(p)=J(F)(q)=
J_{1,1}(F)(p)J_{1,2}(F)(q)-J_{1,2}(F)(p)J_{1,1}(F)(q)=0\}$. One can compute that $X$ has codimension at least $5$, thus the projection of $X$ on $\Omega_2^*(d_1,d_2)$ has empty fibers on some open subset $U\subset \Omega_2^*(d_1,d_2)$.
\end{proof}

Hence for a general $F$ the only singularities of $\Delta(F)$ are cusps and nodes. We showed in Theorem \ref{thmcusps} that there are exactly $c(F)=d_1^2+d_2^2+3d_1d_2-6d_1-6d_2+7$ cusps. Now we will compute the number $d(F)$ of nodes of $\Delta(F)$. We will use the following theorem of Serre (see \cite{mil}, p. 85):

\begin{theo}\label{thmgenusdelta}
If $\Gamma$ is an irreducible curve of degree $d$ and genus $g$  in the complex projective plane
then $$\frac{1}{2} (d-1)(d-2)= g + \sum_{z\in \Sing(\Gamma)} \delta_z,$$
where $\delta_z$ denotes the delta invariant of a point $z$.
\end{theo}

First we compute the degree of the discriminant:

\begin{lem}\label{lemdegdisc}
Let $F=(f,g)\in \Omega(d_1,d_2)$ be a general mapping. If $d_1\ge d_2$ then $\deg\Delta(F)= d_1(d_1+d_2-2)$.
\end{lem}

\begin{proof}
Let $L\subset \C^2$ be a generic line $\{ax+by+c=0\}$. Then $L$ intersects $\Delta(F)$ in smooth points and $\deg\Delta(F)=\# L\cap \Delta(F)$.
If $j: C(F)\to \Delta(F)$ is a mapping induced by $F$ then $\# L\cap \Delta(F)=\# j^{-1}(L\cap \Delta(F)).$ The curve $j^{-1}(L)=\{af+bg+c=0\}$ has no common points at infinity
with $C(F)$. Hence by Bezout Theorem we have $\# j^{-1}(L\cap \Delta(F))=(\deg j^{-1}(L))(\deg C(F))=d_1(d_1+d_2-2)$. Consequently $\deg \Delta(F) = d_1(d_1+d_2-2)$.
\end{proof}

We have the following method of computing the delta invariant in the case of one analytic branch (see \cite{mil}, p. 92-93):

\begin{theo}\label{milnor2}
Let $V_0\subset \C^2$ be an irreducible germ of an analytic curve with the Puiseux parametrization of the form
 $$z_1=t^{a_0},\ z_2=\sum_{i>0} \lambda_i t^{a_i}, \text{ where }\lambda_i\neq 0,\ a_1<a_2<a_3<\ldots$$
Let $D_j=\gcd(a_0,a_1,\ldots, a_{j-1}).$ Then $$\delta_0=\frac{1}{2}\sum_{j\ge 1} (a_j-1)(D_j-D_{j+1}).$$
If $V=\bigcup^r_{i=1} V_i$ has $r$ branches then $$\delta(V)=\sum^r_{i=1} \delta(V_i)+\sum_{i<j} V_i\cdot V_j,$$
where $V\cdot W$ denotes the intersection product.
\end{theo}

The main result of this section will be based on the following:

\begin{theo}\label{theodeltaz}
Let $F\in \Omega(d_1,d_2)$ be a general mapping. Let $d_1\geq d_2$ and $d=\gcd(d_1,d_2)$. Denote by $\overline{\Delta}$ the projective closure
of the discriminant $\Delta$. Then $$\sum_{z\in (\overline{\Delta}\setminus \Delta)} \delta_z= \frac{1}{2}d_1(d_1-d_2)(d_1+d_2-2)^2+\frac{1}{2}(-2d_1+d_2+d)(d_1+d_2-2).$$
\end{theo}

\begin{proof}
Let $\tilde{f}(x,y,z)=z^{d_1}f\left(\frac{x}{z},\frac{y}{z}\right)$ and $\tilde{g}(x,y,z)=z^{d_2}g\left(\frac{x}{z},\frac{y}{z}\right)$ be the homogenizations of $f$ and $g$ and let $\overline{f}(x,z)=\tilde{f}(x,1,z)$ and $\overline{g}(x,z)=\tilde{g}(x,1,z)$. For a general mapping the curves $C(F)$ and $\{f=0\}$ have no common points at infinity (see Lemma \ref{lemstyczne}). Moreover by a linear change of coordinates in the domain we can ensure that $(1:0:0)\notin \overline{C(F)}$. Thus $F$ extends to a neighborhood of $\overline{C(F)}\cap L_\infty$ on which it is given by the formula
$$\overline{F}(x,z)= \left(z^{d_1-d_2}\frac{\overline{g}(x,z)}{\overline{f}(x,z)}, \frac{z^{d_1}}{\overline{f}(x,z)}\right).$$

Let $\{P_1,\ldots,P_{d_1+d_2-2}\}=\overline{C(F)}\cap L_\infty$, fix a point $P=P_i$. The curve $\overline{C(F)}$ is transversal to the line at infinity so it has a local parametrization at $P$ of the form $\gamma(t):=(\sum_i a_it^i,t)$. We have the following:

\begin{lem}\label{lemstyczne}
If $F$ is a general mapping then $\overline{f}(P)\neq 0$, $\overline{g}(P)\neq 0$ and
$$\overline{f}(\gamma(t))=\overline{f}(P)(1+ct+\ldots),\ \overline{g}(\gamma(t))=\overline{g}(P)(1+dt+\ldots),$$
where $cd\neq 0$ and $d_2c\neq d_1d$.
\end{lem}

\begin{proof}
Let $L_\infty$ be the line at $\{z=0\}$ in $\mathbb{P}^2$ and let $\tilde{J}$ be the homogenization of $J$. Obviously $\tilde{J}(F)=\pa{\tilde{f}}{x}\pa{\tilde{g}}{y}-\pa{\tilde{f}}{y}\pa{\tilde{g}}{x}$. To prove that $\overline{f}(P)\neq 0$ consider the set $X=\{(p,F)\in L_\infty\times\Omega_2(d_1,d_2):\ \tilde{f}(p)=\tilde{J}(p)=0\}$. Clearly $X$ has codimension $2$ in $L_\infty\times\Omega_2(d_1,d_2)$ so a general fiber of the projection on $\Omega_2(d_1,d_2)$ is empty. Similarly we obtain $\overline{g}(P)\neq 0$.

Now let $\overline{J}(x,z)=\tilde{J}(x,1,z)$. Since $\overline{J}(\gamma(t))=0$ and $\pa{\gamma(t)}{t}_{|t=0}=(a_1,1)$ we have
$$\overline{f}(P)c\pa{\overline{J}(P)}{x}=\left(\pa{\overline{f}(\gamma(t))}{t}\pa{\overline{J}(\gamma(t))}{x}\right)_{|t=0}=$$
$$\left(\pa{\overline{f}(\gamma(t))}{x}a_1\pa{\overline{J}(\gamma(t))}{x}+\pa{\overline{f}(\gamma(t))}{z}\pa{\overline{J}(\gamma(t))}{x}\right)_{|t=0}=$$
$$\left(-\pa{\overline{f}(\gamma(t))}{x}\pa{\overline{J}(\gamma(t))}{z}+\pa{\overline{f}(\gamma(t))}{z}\pa{\overline{J}(\gamma(t))}{x}\right)_{|t=0}=
\pa{\overline{f}(P)}{z}\pa{\overline{J}(P)}{x}-\pa{\overline{f}(P)}{x}\pa{\overline{J}(P)}{z}$$

Consider the set
$$X=\left\{(p,F)\in L_\infty\times\Omega_2(d_1,d_2):\ \tilde{J}(F)(p)=
\left(\pa{\tilde{f}}{z}\pa{\tilde{J}}{x}-\pa{\tilde{f}}{x}\pa{\tilde{J}}{z}\right)(p)=0\right\}.$$

Note that if $c=0$ then the fiber over $P$ of the projection from $X$ to $L_\infty$ is non-empty. Hence it suffices to prove that $X$ has codimension at least $2$. Similarly as in the proof of Lemma \ref{bir} we take $p=(0:1:0)$ and $Y:=\{p\}\times\Omega_2(d_1,d_2)$ and show that $X\cap Y$ has codimension $2$ in $Y$. Let $a_{ij}$ be the parameters in $\Omega_2(d_1,d_2)$ giving the coefficients of $\tilde{f}$ at $x^iy^{d_1-i-j}z^j$ (i.e. of $f$ at $x^iy^{d_1-i-j}$) and let $a_{ij}$ describe the coefficients of $\tilde{g}$.

The first equation of $X\cap Y$ is $d_2a_{01}b_{00}-d_1b_{01}a_{00}$, denote the second one by $w$. We have $\pa{w}{a_{02}}=2d_2b_{00}a_{10}$, $\pa{w}{b_{02}}=-2d_1a_{00}a_{10}$, $\pa{w}{a_{11}}=d_2b_{00}a_{01}$ and $\pa{w}{b_{11}}=d_1a_{00}a_{01}$, so the equations are independent outside the set $\{a_{10}=a_{01}=0\}\cup\{a_{00}=b_{00}=0\}$. Thus $X\cap Y$ has codimension 2 in $Y$.

Finally note that if $d_2c=d_1d$ then $$d_2\overline{g}(P)\left(\pa{\overline{f}}{z}\pa{\overline{J}}{x}-\pa{\overline{f}}{x}\pa{\overline{J}}{z}\right)(P)=
d_1\overline{f}(P)\left(\pa{\overline{g}}{z}\pa{\overline{J}}{x}-\pa{\overline{g}}{x}\pa{\overline{J}}{z}\right)(P).$$ Hence we consider the set
$$X=\Big\{(p,F)\in L_\infty\times\Omega_2(d_1,d_2):\ \tilde{J}(F)(p)=$$
$$d_2\tilde{g}(p)\left(\pa{\tilde{f}}{z}\pa{\tilde{J}}{x}-\pa{\tilde{f}}{x}\pa{\tilde{J}}{z}\right)(p)-
d_1\tilde{f}(p)\left(\pa{\tilde{g}}{z}\pa{\tilde{J}}{x}-\pa{\tilde{g}}{x}\pa{\tilde{J}}{z}\right)(p)=0\Big\}.$$
Similarly as above one can show that it has codimension $2$, which concludes the proof.
\end{proof}

Let $C_p$ be the branch of $\overline{C(F)}$ at $P$. We find the Puiseux expansion of the branch $\overline{F}(C_P)$ of $\overline{\Delta(F)}$ at $\overline{F}(P)$.  We have
$$\overline{F}(\gamma(t))=\left(t^{d_1-d_2}\frac{\overline{g}(\gamma(t))}{\overline{f}(\gamma(t))},\frac{t^{d_1}}{\overline{f}(\gamma(t))}\right)=$$
$$\left(t^{d_1-d_2}(1+(d-c)t+\ldots)\frac{\overline{g}(P)}{\overline{f}(P)},\frac{t^{d_1}(1-ct+\ldots)}{\overline{f}(P)}\right).$$

If $d_1=d_2$ then by Lemma \ref{lemstyczne} we have $d-c\neq 0$ and $\overline{F}(C_P)$ is smooth at $\overline{F}(P)$. So assume $d_1>d_2$. Since the function $h(t)=\left(\frac{\overline{f}(P)}{\overline{g}(P)}\frac{\overline{g}(\gamma_P(t))}{\overline{f}(\gamma_P(t))}\right)^{\frac{1}{d_1-d_2}}=
1+\frac{d-c}{d_1-d_2}t+\ldots$ is invertible in $t=0$ we can introduce a new variable $T=th(t)$. We have $\overline{F}(\gamma(T))=\left(T^{d_1-d_2}\frac{\overline{g}(P)}{\overline{f}(P)},T^{d_1}h(t)^{-d_1}(1-ct+\ldots)\frac{1}{\overline{f}(P)}\right)$.
Moreover $h(t)^{-d_1}(1-ct+\ldots)=(1-d_1\frac{d-c}{d_1-d_2}T+\ldots)(1-cT+\ldots)=1+\frac{d_2c-d_1d}{d_1-d_2}T+\ldots$. By Lemma \ref{lemstyczne} we have $d_2c-d_1d\neq 0$ and we can apply Theorem \ref{milnor2} to compute $\delta(C_P)_{\overline{F}(P)}$. Since $a_0=d_1-d_2$, $a_1=d_1$ and $a_2=d_1+1$, we have $2\delta(C_P)_{\overline{F}(P)}=(d_1-1)(d_1-d_2-d)+(d_1+1-1)(d-1)=(d_1-1)(d_1-d_2-1)+(d-1)$.

To proceed further we also need:

\begin{lem}\label{row}
If $F$ is a general mapping then
$$\overline{f}(P_i)^{d_2}\overline{g}(P_j)^{d_1}\neq\overline{f}(P_j)^{d_2}\overline{g}(P_i)^{d_1}$$
for $i,j\in\{1,2,\ldots,d_1+d_2-2\}$ and $i\neq j$.
\end{lem}

\begin{proof}
Consider the set $X=\{(p,q,F)\in L_\infty\times L_\infty\times\Omega_2(d_1,d_2):\ p\neq q,\ \tilde{J}(F)(p)=\tilde{J}(F)(p)=
\tilde{f}(p)^{d_2}\tilde{g}(q)^{d_1}-\tilde{f}(q)^{d_2}\tilde{g}(p)^{d_1}=0\}$. Since one can prove similarly as in Lemma \ref{lemstyczne} that $X$ has codimension $3$, there is a dense open subset $S\subset \Omega(d_1,d_2)$ such that the projection from $X$ has empty fibers over $F\in S$.
\end{proof}

Now we are in a position to compute $\sum_{z\in (\overline{\Delta}\setminus \Delta)} \delta_z$. If $d_1=d_2$ then $\overline{\Delta}$ has exactly $d_1+d_2-2$ smooth points at infinity and consequently $\sum_{z\in (\overline{\Delta}\setminus \Delta)} \delta_z=0$. So assume $d_1>d_2$, then $\overline{\Delta}$ has only one point at infinity $Q=(1:0:0)$.
In $Q$ the curve $\overline{\Delta}$ has exactly $r=d_1+d_2-2$ branches $V_i=\overline{F}(C_{P_i})$. We computed above that $2\delta(V_i)_Q=(d_1-1)(d_1-d_2-1)+(d-1)$. Now we will compute $V_i\cdot V_j$.
Let $t_{a,b}(x,y)=(x+a,y+b)$. By the dynamical definition of intersection there exists a neighborhood $U$ of $0$, such that for small general $a,b$ we have
$$V_i\cdot V_j=\# (U\cap V_i\cap t_{a,b}(V_j)).$$
This means that $V_i\cdot V_j$ is equal to the number of solutions of the following system:
$$\frac{\overline{g}(P_i)}{\overline{f}(P_i)}T^{d_1-d_2}=\frac{\overline{g}(P_j)}{\overline{f}(P_j)}S^{d_1-d_2}+a,$$
$$\frac{1}{\overline{f}(P_i)}T^{d_1}(1+\alpha_iT+\ldots)=\frac{1}{\overline{f}(P_j)}S^{d_1}(1+\alpha_jS+\ldots)+b,$$
where $a,b$ and $S,T$ are sufficiently small. Take
$$Q:(\C^2,0)\to (\C^2,0),$$
$$Q(T,S)=\left(\frac{\overline{g}(P_i)}{\overline{f}(P_i)}T^{d_1-d_2}-\frac{\overline{g}(P_j)}{\overline{f}(P_j)}S^{d_1-d_2},
\frac{1}{\overline{f}(P_i)}T^{d_1}(1+\alpha_iT+\ldots)-\frac{1}{\overline{f}(P_j)}S^{d_1}(1+\alpha_jS+\ldots)\right).$$
Thus we have $V_i\cdot V_j=\mult_0 Q$. Note that by Lemma \ref{row} the minimal homogenous polynomials of the two components of $Q$ have no nontrivial common zeroes, hence $V_i\cdot V_j=d_1(d_1-d_2)$. Consequently
$$\sum_i \delta(V_i)+ \sum_{i>j} V_i\cdot V_j=\frac{1}{2}[(d_1-1)(d_1-d_2-1)+(d-1)](d_1+d_2-2)+$$
$$\frac{1}{2}d_1(d_1-d_2)(d_1+d_2-2)(d_1+d_2-3)=$$
$$\frac{1}{2}d_1(d_1-d_2)(d_1+d_2-2)^2+\frac{1}{2}(-2d_1+d_2+d)(d_1+d_2-2).$$
\end{proof}

We can now prove the following:

\begin{theo}
There is a Zariski open, dense subset $U\subset \Omega_2(d_1,d_2)$
such that for every mapping $F\in U$ the discriminant $\Delta(F)=F(C(F))$ has only cusps and nodes as singularities. Let $d=\gcd(d_1,d_2)$.
Then the number of cusps is equal to
$$c(F)=d_1^2+d_2^2+3d_1d_2-6d_1-6d_2+7$$
and the number of nodes is equal to
$$d(F)=\frac{1}{2}\left[(d_1d_2-4)((d_1+d_2-2)^2-2)-(d-5)(d_1+d_2-2)-6\right].$$
\end{theo}
\begin{proof}
Let $d_1\geq d_2$ and $D=d_1+d_2-2$. By Lemma \ref{lemdegdisc} we have $\deg\Delta(F)= d_1D$. From Lemma \ref{bir} we know that $\Delta(F)$ has only cusps and nodes as singularities and is birational with $C(F)$. Hence $\Delta(F)$ has genus $g=\frac{1}{2}(D-1)(D-2)$. Thus by Theorem \ref{thmgenusdelta} we have
$$\frac{1}{2} (d_1D-1)(d_1D-2)=\frac{1}{2}(D-1)(D-2)+c(F)+d(F)+ \sum_{z\in (\overline{\Delta}\setminus \Delta)} \delta_z.$$
 Substituting
$$\sum_{z\in (\overline{\Delta}\setminus \Delta)} \delta_z=\frac{1}{2}d_1(d_1-d_2)D^2+\frac{1}{2}(-2d_1+d_2+d)D$$
from Theorem \ref{theodeltaz} we obtain
$$2(c(F)+d(F))=d_1d_2D^2-D^2+3D-d_1D-d_2D-dD=(d_1d_2-2)D^2-(d-1)D.$$
Thus by Theorem \ref{thmcusps} we get:
$$d(F)=\frac{1}{2}\left[(d_1d_2-2)D^2-(d-1)D-2(D^2-2D+d_1d_2-1)\right]=$$
$$\frac{1}{2}\left[(d_1d_2-4)(D^2-2)-(d-5)D-6\right].$$
\end{proof}

\section{Generalized cusps}\label{secGC}

In this section our aim is to estimate the number of cusps of
non-generic mappings. We start from:

\begin{defi}\label{dfGenCus}
Let $F : (\C^2,a)\to (\C^2,F(a))$ be a holomorphic mapping. We say
that $F$ has a generalized cusp at $a$ if  $F_a$ is proper, the
curve $J(F)=0$ is reduced near $a$ and  the discriminant of $F_a$
is not smooth at $F(a)$.
\end{defi}

\begin{re}
If $F_a$ is proper, $J(F)=0$ is reduced near $a$ and $J(F)$ is
singular at $a$ then it follows from Corollary 1.11 from \cite{jel} that
also the discriminant of $F_a$ is singular at $F(a)$ and hence
$F$ has a generalized cusp at $a$.
\end{re}

Now we introduce the index of generalized cusp:

\begin{defi}\label{dfGenCusIn}
Let $F=(f,g): (\C^2,a)\to (\C^2,F(a))$ be a holomorphic mapping.
Assume that $F$ has a generalized cusp at a point $a\in \C^2$.
Since the curve $J(F)=0$ is reduced near $a$, we have that the set
$\{\nabla f=0\}\cap \{\nabla g=0\}$ has only isolated points near
$a.$ For  a  general linear mapping $T\in GL(2)$,  if
$F'=(f',g')=T\circ F$ then $\nabla f'$ does not vanish
identically on any branch of $\{J(F)=0\}$ near $a$.  We say that
the  cusp of $F$ at $a$ has an index $\mu_a:={\dim}_\C {\mathcal
O}_a/(J(F'), J_{1,1}(F'))-{\dim}_\C {\mathcal O}_a/(f'_x, f'_y)$.
\end{defi}

\begin{re}
We show below that the index $\mu_a$ is well-defined and
finite. Moreover, it is easy to see that a simple cusp has index
one.
\end{re}

\begin{re}
Using the exact sequence $1.7$ from \cite{gm1} we see that
$$\mu_a={\dim}_\C {\mathcal O}_a/(J(F), J_{1,1}(F), J_{1,2}(F)).$$
Hence our index coincides with the classical local number of cusps
defined e.g. in \cite{gm1}.
\end{re}

We have (compare with  \cite{gm1}, \cite{gm2}, \cite{gaf}):

\begin{theo}
Let $F=(f,g)\in \Omega_2(d_1,d_2)$ and assume that $F$ has a
generalized cusp at $a\in \C^2$.   If $U_a$ is a sufficiently
small ball around $a$ then $\mu_a$ is equal to the number of
simple cusps in $U_a$ of a general mapping $F'\in \Omega_2(d_1',
d_2')$, where $d_1'\ge d_1, d_2'\ge d_2$, which is sufficiently
close to $F$ in the natural topology of $\Omega_2(d_1', d_2')$.
\end{theo}

\begin{proof}
We can assume that  $\nabla f$ does not vanish identically on any
branch of $\{J(F)=0\}$ near $a.$ In particular we have ${\rm dim}\
{\mathcal O_a}/(f_x,f_y)= {\rm dim } \ {\mathcal O_a}/(J(F), f_x,
f_y)<\infty.$

Let $F_i=(f_i, g_i)\in\Omega_2(d_1', d_2')$ be a sequence of
general mappings, which is convergent to $F.$ Consider the
mappings $\Phi=(J(F), J_{1,1}(F))$, $\Phi_i=(J(F_i),
J_{1,1}(F_i))$, $\Psi=(\nabla f)$ and $\Psi_i=(\nabla f_i).$ Thus
$\Phi_i\to \Phi$ and $\Psi_i\to \Psi.$

Since $a$ is a cusp of $F$ we have $\Phi(a)=0$. Moreover
$d_a(\Phi)<\infty,$ where $d_a(\Phi)$ denotes the local
topological degree of $\Phi$ at $a$. Indeed, if $J_{1,1}(F)=0$ on
some branch $B$ of the curve $J(F)=0$ then the rank of $F_{|B}$
would be zero and by Sard theorem $F$ has to contract $B$, which
is a contradiction ($F_a$ is proper). By the Rouche Theorem we
have that for large $i$ the mapping $\Phi_i$ has exactly
$d_a(\Phi)$ zeroes in $U_a$ and $\Psi_i$ has exactly
$d_a(\Psi)$ zeroes in $U_a$ (counted with multiplicities, if
$\Psi(a)\not=0$ we put $d_a(\Psi)=0$). However, the mappings $F_i$
are general, in particular all zeroes of $\Phi_i$ and $\Psi_i$ are
simple. Moreover the zeroes of $\Phi_i$ which are not cusps of $F_i$ are
zeroes of $\Psi_i$.
Hence $\mu_a=d_a(\Phi)-d_a(\Psi)$ is indeed the number of
simple cusps of $F_i$ in $U_a$.
\end{proof}

\begin{co}
Let $F\in \Omega_2(d_1,d_2)$. Assume that $F$ has generalized
cusps at points $a_1,\ldots, a_r$. Then $\sum^r_{i=1} \mu_{a_i}\le
d_1^2+d_2^2+3d_1d_2-6d_1-6d_2+7.$
\end{co}


\begin{thebibliography}{10}

\bibitem{fi} T. Fukuda, G. Ishikawa,
{\em On the number of cusps of stable perturbations of a plane-to-plane singularity},
Tokyo J. Math. 10 (1987), no. 2, 375-384.

\bibitem{gm1} T. Gaffney, D.M.Q. Mond, {\em Cusps and double folds of germs of analytic maps $\mathbb C^2 \to \mathbb C^2$},
 J. London Math. Soc. (2) 43 (1991), no. 1, 185-192.

\bibitem{gm2} T. Gaffney, D.M.Q. Mond, {\em Weighted homogeneous maps from the plane to the plane}, Math. Proc. Cambridge Philos. Soc. 109 (1991), no. 3, 451-470.

\bibitem{gaf} T. Gaffney,  {\em Polar multiplicities and equisingularity of map germs}, Topology 32 (1993), no. 1, 185--223.

\bibitem{gun}  R. Gunning, H. Rossi, {\em  Analytic functions of
several variables}, Prentice Hall (1965).

\bibitem{gg} M. Golubitsky, V. Guillemin, {\em  Stable mappings and
their singularities}, GTM, Springer-Verlag  (1973).

\bibitem{jel} Z. Jelonek, {\em On finite regular and holomorphic mappings},
arXiv:1404.7466v3 [math.AG], 2014.

\bibitem{jel1} Z. Jelonek, {\em On semi-equivalence of generically-finite polynomial
mappings}, Math. Z., (2016)
DOI 10.1007/s00209-015-1591-8.


\bibitem{ks} I. Krzy\.zanowska, Z. Szafraniec, {\em On polynomial mappings from the plane to the plane},     J. Math. Soc. Japan Vol. 66, no 3 (2014), 805-818.


\bibitem{mil} J. Milnor, {\em  Singular points of complex hypersurfaces}, Annals of Mathematics Studies, Princeton University Press,  (1968).

\bibitem{wh} H. Whitney, {\em On singularities of mappings of Euclidean spaces. I. Mappings of the plane into
the plane,} Ann. of Math. (2) 62 (1955), 374-410.




    \end{thebibliography}
    \end{document}